\documentclass[10pt]{amsart}
\usepackage{amssymb}
\usepackage{amsmath,amssymb,amsfonts,amsthm,graphics,
latexsym, amscd, amsfonts, epsfig, eepic,epic,psfrag}
\usepackage{mathrsfs}
\usepackage{color}
\usepackage{eucal}
\usepackage{eufrak}
\usepackage[all]{xypic}
\usepackage{xspace}
\usepackage{algorithm}
\usepackage{algorithmic}
\usepackage{url}

\theoremstyle{plain}
\newtheorem{theorem}{Theorem}
\newtheorem{corollary}{Corollary}

\newtheorem{lemma}{Lemma}
\newtheorem{proposition}{Proposition}

\theoremstyle{definition}
\newtheorem{definition}{Definition}

\theoremstyle{example}

\theoremstyle{remark}

\numberwithin{equation}{section}

\begin{document}

\title{Shapes of topological RNA structures}

\author{ Fenix W.D. Huang \and Christian M.~Reidys$^{\star}$
}


\address{Department of Mathematics and Computer Science,
              University of Southern Denmark,
              Campusvej 55, DK-5230
             Odense M, Denmark\\
              Tel.$^*$: +45-24409251\\
              Fax$^*$: +45-65502325\\
              E-mail$^*$:  duck@santafe.edu
}
\date{Received: date / Accepted: date}

\maketitle

\begin{abstract}
A topological RNA structure is derived from a diagram and its shape is
obtained by collapsing the stacks of the structure into single arcs and
by removing any arcs of length one. Shapes contain key topological,
information and for fixed topological genus there exist only finitely
many such shapes.
We shall express topological RNA structures as unicellular maps,
i.e.~graphs together with a cyclic ordering of their half-edges.
In this paper we prove a bijection of shapes of topological RNA
structures. We furthermore derive a linear time algorithm generating
shapes of fixed topological genus.
We derive explicit expressions for the coefficients of the generating
polynomial of these shapes and the generating function of RNA structures
of genus $g$. Furthermore we outline how shapes can be used in order to
extract essential information of RNA structure databases.
\end{abstract}


\section{Introduction} \label{S:Into}


Pseudoknots have long been known as important structural elements in RNA
\cite{Westhof:92a}. These cross-serial interactions between RNA nucleotides
are functionally important in tRNAs, RNaseP \cite{Loria:96a}, telomerase
RNA \cite{Staple:05}, and ribosomal RNAs \cite{Konings:95a}.
Pseudoknots in plant virus RNAs mimic tRNA structures, and {\it in vitro}
selection experiments have produced pseudoknotted RNA families that bind
to the HIV-1 reverse transcriptase \cite{Tuerk:92}.

Since the prediction of general RNA pseudoknot structures is NP-complete
\cite{Lynsoe:00}, one frequently sticks to certain subclasses of pseudoknots,
suitable for the dynamic programming paradigm \cite{Rivas:99, Reidys:11a}.

In \cite{Reidys:11a} a folding algorithm, {\tt gfold}, for one such class of
RNA structures has been presented. This class consists of structures of fixed
topological genus. The topological filtration of RNA structures has first
been proposed by Penner and Waterman in \cite{Waterman:93} and later, as an
application of the Matrix model in \cite{Orland:02} and \cite{Bon:08}. In
\cite{Reidys:11a,Reidys:top1} a representation theoretic Ansatz is employed
that traces back to Zagier \cite{Zagier:95}. \cite{Reidys:top1} connects
RNA shapes of fixed topological genus with Riemann's moduli space.

RNA structures are represented as diagrams, that is as labeled graphs
over the vertex set $[n]=\{1, \dots, n\}$ with vertex degrees $\le 3$,
represented by drawing its vertices on a horizontal line and its arcs
$(i,j)$ ($i<j$), in the upper half-plane, see Figure~\ref{F:struc-shape} (A).
We assume the vertices to be connected by the edges $\{i,i+1\}$, $1\le i<n$,
which are not considered to be arcs (but contribute to a nodes's degree).
Furthermore, vertices and arcs correspond to the nucleotides {\bf A}, {\bf G},
{\bf U} and {\bf C} and Watson-Crick base pairs ({\bf A-U}, {\bf G-C}) or
wobble base pairs ({\bf U-G}), respectively.
Considering only the Watson-Crick and wobble base pair RNA structures, we
set the restriction that one vertex can only paired with at most another
vertex. Let $i<r$, we call arcs $(i,j)$ and $(r,s)$ crossing if $i<r<j<s$ holds.
In this representation a pseudoknot-free secondary structure is a diagram
without crossing arcs. Otherwise, i.e.~diagrams with crossings represent
pseudoknot structures.
The above mentioned topological folding algorithm, {\tt gfold}, depends
crucially on RNA shapes. These are obtained
(recursively) by collapsing stacks into arcs and by removing any $1$-arc.
Shapes are obtained by considering homotopy-classes of arcs and represent
thereby the ``key'' topological information that lies within the original
structure, see Figure~\ref{F:struc-shape} (B). RNA shapes are the central
determinant of the multiple-context free language of topological RNA
structures.

\begin{figure}[ht]
\begin{center}
\includegraphics[width=0.8\columnwidth]{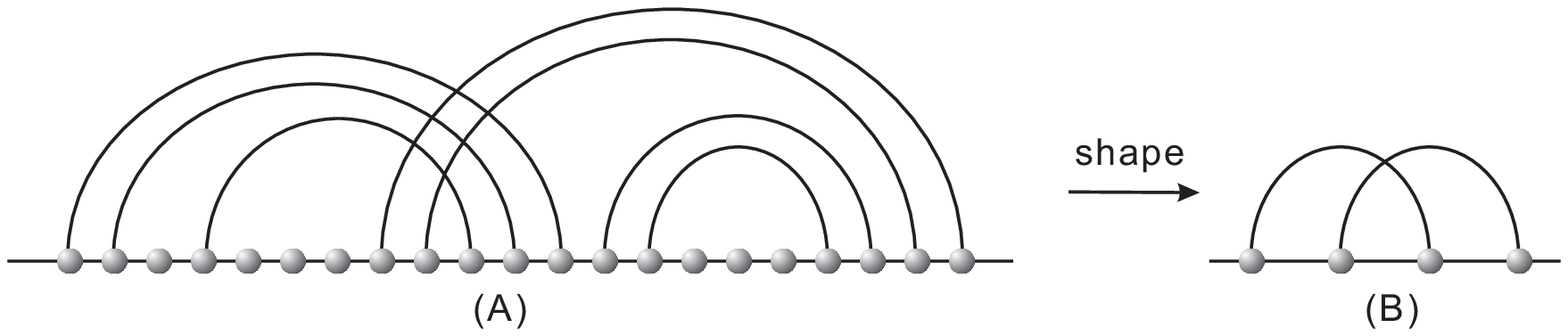}
\end{center}
\caption{\small From a diagram (A) to its shape (B).
}
\label{F:struc-shape}
\end{figure}

In \cite{Reidys:11a} it identifies a particular topological fact, crucial for
folding. That is, for fixed topological genus, there exist only finitely
many shapes. This immediately implies that, despite the fact that there
are infinitely many RNA structures of fixed topological genus, the
generating function can be reduced to a generating polynomial. We shall
refer to this polynomial as the shape polynomial. While the situation is
fairly easy for genus one \cite{Reidys:11a}, see Figure~\ref{F:4shape}, for
higher genera it is not trivial to obtain the shapes.

\begin{figure}[ht]
\begin{center}
\includegraphics[width=0.8\columnwidth]{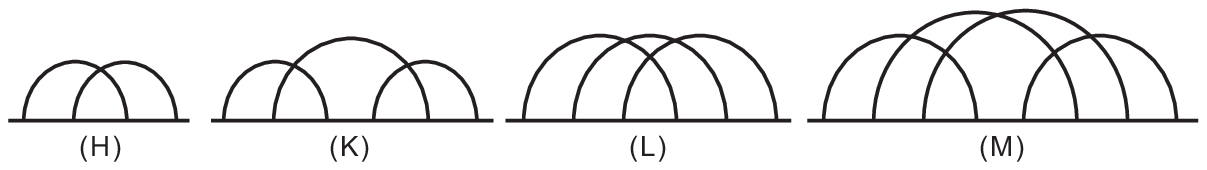}
\end{center}
\caption{\small The four shapes employed in the topological folding algorithm {\tt gfold}.
}
\label{F:4shape}
\end{figure}

Interestingly, more than $95\%$ of all known RNA-pseudoknot structures
are build very ``regularly''. They are derived from the aforementioned $4$
shapes by means of concatenation and nesting.
This observation has led to the notion of $\gamma$-structures
\cite{Han:13}, obtained as concatenation and nesting of shapes of
genus less than $\gamma$. Thus, despite the fact that the overall genus of
$\gamma$-structures is arbitrary, they are composed by finitely many blocks of
at most genus $\gamma$-complexity.

This fits well with what we know about RNA secondary structures: these are build by
concatenation
and nesting of simple arcs. Topological RNA Structures generalize this in
a natural way, utilizing novel building blocks, more complex than simple arcs,
i.e.~RNA shapes described in the following.
The problem is thus reduced to finding and analyzing shapes, whose numbers
increase rapidly with increasing genus, see Table~\ref{T:shapes}.

\begin{table}
\begin{center}
\begin{tabular}{c|ccccc}
 & g=1 & 2 & 3 & 4 & 5 \\
\hline
   & 4 & 3696 & 15214144 & 148120104704 & 2638025019442176 \\
\end{tabular}
\end{center}
\caption{\small The number of shapes of fixed genus $g$.
}\label{T:shapes}
\end{table}

Recently, a linear time uniform random sampler for pseudoknotted RNA structures of
given topological genus has been presented \cite{Huang:13}. Unfortunately this
framework cannot be used for shapes.

In this paper we present a linear time, uniform sampling algorithm for shapes of fixed
topological genus. The core idea traces back to a bijection of Chapuy
\cite{Chapuy:11}, that reduces genus by recursive ``splicing'' of certain
vertices. In difference to the aforementioned uniform sampling \cite{Huang:13},
the work is based on a specific refinement. Namely, here we keep track
of the labeled vertices produced by slicing over many such processes. This enhancement
enables us to establish new recursions, which allow us to uniformly generate shapes
from trees with a specific number of labeled vertices. The process requires us to
characterize which trees actually generate shapes (shape-trees). To this end we show that
a bijection of R\'{e}my \cite{Remy:85} is compatible with shape-trees and can therefore
be restricted. As a result we can give an explicit formula for the coefficients of the
shape polynomial.

The paper is organized as follows: in Section~\ref{S:basic} we introduce the basic
framework. In Section~\ref{S:ucmap} we give an interpretation of the generating
function of structures of fixed topological genus based on our refined splicing.
The result implies a formula for the $P_g(z)$ polynomials of
\cite{Harer:86,Reidys:top1}. Finally we study the recursion for shapes
in Section~\ref{S:shape}. Here we show that first the original maps restrict
naturally to shapes and secondly that R\'{e}my's bijection \cite{Remy:85} can
be restricted to shape-trees. These two observations allow us to find an explicit
formula for the shape polynomial. Finally, in Section~\ref{S:Gen}, we translate
the results of Section~\ref{S:shape} and derive the linear time, uniform generation
algorithm for shapes of fixed topological genus.

\section{Some basic facts} \label{S:basic}

\subsection{Diagrams}

A diagram is a labeled graph over the vertex set $[n]=\{1, \dots, n\}$ in
which each vertex has degree $\le 3$, represented by drawing its vertices
in a horizontal line. The backbone of a diagram is the sequence of
consecutive integers $(1,\dots,n)$ together with the edges $\{\{i,i+1\}
\mid 1\le i\le n-1\}$. The arcs of a diagram, $(i,j)$, where $i<j$, are
drawn in the upper half-plane. We shall distinguish backbone edges
$\{i,i+1\}$ from arcs $(i,i+1)$, which we refer to as a $1$-arc.
Two arcs $(i,j)$, $(r,s)$, where $i<r$ are crossing if $i<r<j<s$ holds.
Parallel arcs of the form $\{(i,j), (i+1,j-1), \cdots, (i+\ell-1, j-\ell+1)\}$
is called a stack, and $\ell$ is called the length of this stack.
Furthermore, the particular arc, $(1,n)$, is called the rainbow, see
Figure~\ref{F:bc} (A).

\begin{figure}[ht]
\begin{center}
\includegraphics[width=0.9\columnwidth]{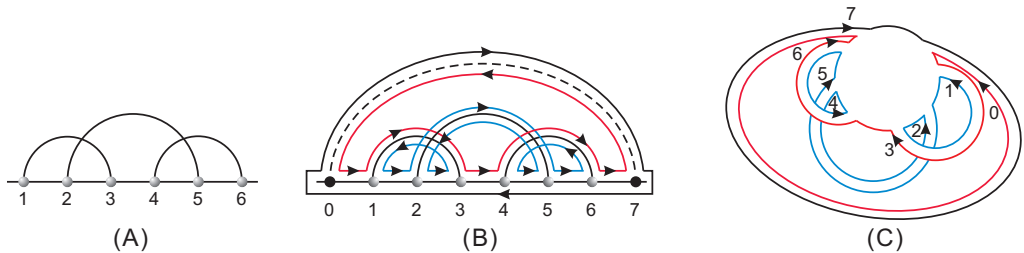}
\end{center}
\caption{\small (A) A diagram.  (B) the fattening of (A) augmented by the
rainbow (0, 7). Here $\sigma=(0,1,2,3,4,5,6,7)$, $\alpha=(0,7)(1,3)(2,5)(4,6)$.
Accordingly $\gamma= \alpha \circ \sigma= (0, 3, 6)(1, 5, 4, 2)(7)$ has two cycles.
(C) Collapsing the backbone into a vertex.
}
\label{F:bc}
\end{figure}


\subsection{Fatgraphs and unicellular maps}
In this section, we discuss the filtration of diagrams by topological genus. In order to
extract topological properties of diagrams those need to be enriched to fatgraphs. The
latter are tantamount to a cell-complex of an by construction orientable, topological
surfaces.
Formally, we make this transition \cite{Reidys:top1} by ``thickening''
the edges of the diagram into (untwisted) bands or ribbons. Furthermore
each vertex is inflated into a disc.
This inflation of edges and vertices means to replace a set of incident edges
by a sequence of half-edges. This constitutes the fatgraph $\mathbb{D}$
\cite{Loebl:08,Penner:10}, see Figure~\ref{F:bc} (B).


A fatgraph is thus a graph enriched by a cyclic ordering of the incident
half-edges at each vertex and consists of the following data: a set of
half-edges, $H$, cycles of half-edges as vertices and pairs of half-edges
as edges. Consequently, we have the following definition:

\begin{definition}
A fatgraph is a triple $(H, \sigma, \alpha)$, where $\sigma$ is
the vertex-permutation and $\alpha$ a fixed-point free involution.
\end{definition}

In the following we will deal with orientable fatgraphs\footnote{Here ribbons may also
be allowed to twist giving rise to possibly non-orientable surfaces \cite{Massey:69}.}.
Each ribbon has two boundaries. The first one in counterclockwise
order shall be labeled by an arrowhead, see Figure~\ref{F:bc} (B).

A fatgraph $\mathbb{D}$ exhibits a phenomenon, not present in its
underlying graph $D$. Namely, one can follow the (directed) sides of the
ribbons rotating counterclockwise around the vertices. This gives rise to
$\mathbb{D}$-cycles or boundary components, constructed by following these
directed boundaries from disc to disc. Algebraically, this amounts to form
the permutation $\gamma=\alpha \circ \sigma$.

In the following we consider only diagrams with a rainbow. As we shall see,
the rainbow arc provides a canonical first boundary component, which travels
on top of the rainbow arc and the bottom of the backbone of the diagram.


A fatgraph, $\mathbb{D}$, can be viewed as a ``drawing'' on a
certain topological surface.
$\mathbb{D}$ is a $2$-dimensional cell-complex over its geometric
realization, i.e.~a surface without boundary, $X_{\mathbb{D}}$, realized
by identifying all pairs of edges \cite{Massey:69}.
Key invariants of the latter, like Euler characteristic \cite{Massey:69}
\begin{eqnarray}\label{E:euler}
\chi(X_\mathbb{D}) & = & v - e + r,\\
\label{E:genus}
g(X_\mathbb{D}) & = & 1-\frac{1}{2}\chi(X_\mathbb{D}),
\end{eqnarray}
where $v,e,r$ denotes the number of discs, ribbons and boundary components
in $\mathbb{D}$ \cite{Massey:69} are defined combinatorially. However,
equivalence of simplicial and singular homology \cite{Hatcher:02} implies that
these combinatorial invariants are in fact invariants of $X_{\mathbb{D}}$ and thus
topological. This means the surface $X_{\mathbb{D}}$ provides a topological
filtration of fatgraphs.

Since adding a rainbow or collapsing the backbone of a diagram, see Figure~\ref{F:bc} (C),
does not change the Euler characteristic, the relation between
genus and number of boundary components is solely determined by
the number of arcs in the upper half-plane:
\begin{equation}\label{E:ee}
2-2g-r = 1-n,
\end{equation}
where $n$ is number of arcs and $r$ the number of boundary
components. The latter can be computed easily and allows us therefore
to obtain the genus of the diagram.

\begin{definition}
A unicellular map $\mathfrak{m}$ of size $n$ is a fatgraph
$\mathfrak{m}(n)=(H,\alpha,\sigma)$ in which the
permutation $\alpha\circ\sigma$ is a cycle of length $2n$.
\end{definition}

While unicellular maps are simply particular fatgraphs, they naturally
arise in the context of diagrams, by two observations. First in the
diagram one may collapse the backbone into a single
vertex. Second the mapping
$$
\pi \colon (H,\sigma, \alpha) \ \mapsto \ (H, \alpha\circ \sigma,\alpha),
$$
is evidently a bijection between fatgraphs having one vertex and
unicellular maps, see Figure~\ref{F:dual}.
The mapping is called the {\it Poincar\'{e} dual} and interchanges
boundary components by vertices, preserving topological genus.
In the following, we use $\pi$ to denote the {Poincar\'{e} dual}.

\begin{figure}[ht]
\begin{center}
\includegraphics[width=0.8\columnwidth]{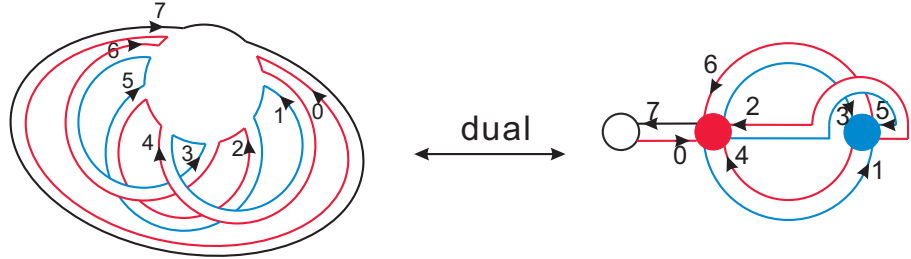}
\end{center}
\caption{\small The Poincar\'{e} dual: we map a fatgraph with $1$
vertex and $3$ boundary components into a fatgraph with $3$ vertexes
and $1$ boundary component.
}
\label{F:dual}
\end{figure}


Given a unicellular map the permutation $\sigma$ and $\gamma$ induces two
linear orders of half-edges
$$
r <_{\gamma} \gamma(r) <_{\gamma} \dots <_{\gamma} \gamma^{2n-1}(r), \quad
r <_{\sigma} \sigma(r) <_{\sigma} \dots <_{\sigma} \sigma^{k}(r).
$$
Let $a_1$ and $a_2$ be two distinct half-edges in $\mathfrak{m}$. Then
$a_1<_{\gamma} a_2$ expresses the fact that $a_1$ appears before $a_2$ in
the boundary component $\gamma=\alpha\circ \sigma$.
Suppose two half-edges $a_1$ and $a_2$ belong to the same vertex $v$. Note
that $v$ is effectively a cycle which we assume to originate with the first
half-edge along which one enters $v$ traveling $\gamma$. Then
$a_1<_\sigma a_2$ expresses the fact that $a_1$ appears (counterclockwise)
before $a_2$.

The Poincar\'{e} dual maps the rainbow into a distinguished vertex of degree one
and provides thereby a natural origin for the cycle $\gamma$. We call this
vertex the {\it plant}, see Figure~\ref{F:dual}.
Given a unicellular map we call a half-edge the minimum half-edge of a
vertex $v$ if it is the first half-edge via which $\gamma$ visits $v$.

\subsection{Shapes}

An arc is called a $1$-arc if it is the form $(i,i+1)$. Two arcs are called
parallel if they are of the form of $(i,j)$ and $(i+1,j-1)$. A diagram is called
a {\it preshape} if it contains neither $1$-arcs nor parallel arcs, see
Figure~\ref{F:shape}. A preshape without a rainbow is called pure. Clearly,
there is a projection from preshapes to pure preshapes obtained by removing the
latter. A shape is then obtained from a pure preshape by adding a rainbow.

\begin{figure}[ht]
\begin{center}
\includegraphics[width=0.8\columnwidth]{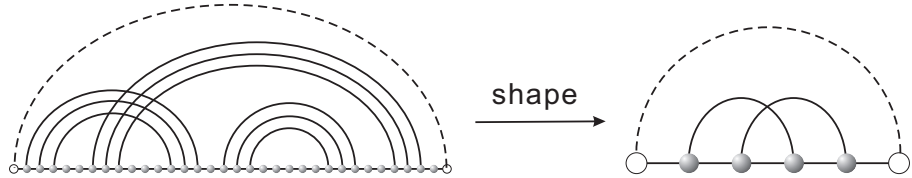}
\end{center}
\caption{\small From a diagram to a shape by removing all $1$-arc and parallel arcs.
The dash arc is a rainbow arc, where a preshape is nested inside.
}
\label{F:shape}
\end{figure}

\begin{proposition}
Let $S_g$ be a shape of genus $g$ having $n$ arcs. Let further $\mathfrak{s}_g$
denote its associated unicellular map. Then any vertex in $\mathfrak{s}_g$ has
degree $\ge 3$.
\end{proposition}
\begin{proof}
We proof the proposition by contradiction. Suppose $v$ is a vertex in
$\mathfrak{s}_g$. The boundary component in $S_g$ associated to $v$
travels $d(v)$ arcs. The Poincar\'{e} dual maps a boundary component to a vertex, so
in case of $d(v)=1$, the boundary component travels only one arc and is thus
a $1$-arc. A boundary component consisting of two arcs is obtained by either
parallel arcs or subsequent arcs, where the endpoint of the second arc travels
via the backbone to the start point of the first. The latter case is impossible
since a shape always contains a rainbow which increases the size to three and
the proposition follows, see Figure~\ref{F:dual}.
\end{proof}

\subsection{Topological induction}

In this section we present a construction of \cite{Chapuy:11}, which plays a
key role for our main result. It consists of two processes: a slicing-map
$\Xi$ and a gluing-map $\Lambda$, which, when restricted to the proper
classes, are inverse to each other, see Figure~\ref{F:glue_slice}.

\begin{figure}[ht]
\begin{center}
\includegraphics[width=0.8\columnwidth]{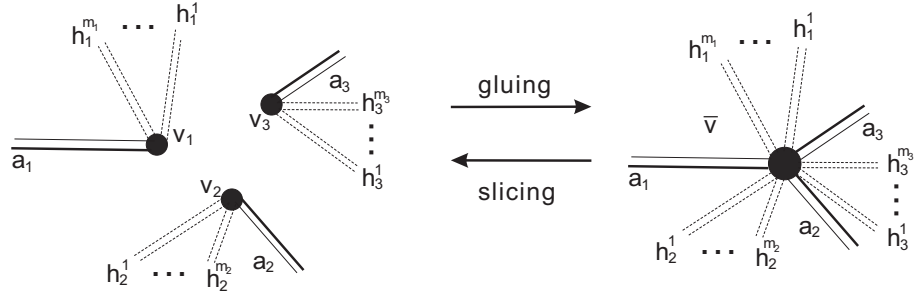}
\end{center}
\caption{\small Illustration of gluing and slicing in a unicellular map.
}
\label{F:glue_slice}
\end{figure}

The slicing process splits a vertex into $(2g+1)$ vertices and thereby
reduces the genus of the map by $g$. Gluing effectively inverts
slicing, namely: gluing any $(2g+1)$ vertices in a unicellular map increases
the genus of the map by $g$. Slicing and gluing preserve unicellularity.

\begin{definition}
A half-edge $h$ is an {\it up-step} if $h<_{\gamma} \sigma(h)$, and a
{\it down-step} if $\sigma(h) \le_{\gamma} h$. $h$ is called a {\it trisection}
if $h$ is a down-step and $\sigma(h)$ is not the minimum half-edge of its
respective vertex.
\end{definition}

The number of trisections in a unicellular map is an invariant of a unicellular map with
fixed genus $g$. Moreover, then number is given by the following lemma:

\begin{lemma} \label{L:trisection}
\cite{Chapuy:11} Let $\mathfrak{m}_g$ be a unicellular map of
genus $g$. Then $\mathfrak{m}_g$ has exactly $2g$ trisections.
\end{lemma}

Slicing reduces the number of trisections in a unicellular map of genus $g$. First we pick up a
trisection $\tau$ and assume it is contained in a vertex $v$.
Let $a_1$ denote the minimum half-edge in $\overline{v}$, and $a_3$ denote the half-edge located
anticlockwise from $\tau$. We consider for the half-edge between $a_1$ and $a_3$,
$a_2$, that is the minimum half-edge satisfying $a_2>_\gamma a_3$. We can always find such a
half-edge $a_2$ since $\tau$ is a trisection and $\tau>_\gamma a_3$, by definition.

Let
$$
\overline{v}=(a_1, h_2^1,\ldots, h_2^{m_2}, a_2, h_3^1,\ldots,h_3^{m_3},a_3,h_1^{1},\ldots,h_1^{m_1}),
$$
and
$$
\overline{\gamma}=(\ell_1^1, \ldots, \ell_{k_1}^1, a_1, \ell_1^3, \ldots, \ell_{k_3}^3, a_3,
\ell_1^2, \ldots, \ell_{k_2}^2, a_2, \ell_1^4, \ldots, \ell_{k_4}^4).
$$
We slice $\overline{v}$ into three vertices $v_i$, $i=1,2,3$, where
$v_i=(a_i, h_i^1,\ldots, h_i^{m_i})$.
The new boundary is given by
$$
\gamma=(\ell_1^1, \ldots, \ell_{k_1}^1, a_1, \ell_1^2, \ldots, \ell_{k_2}^2, a_2,
\ell_1^3, \ldots, \ell_{k_3}^3, a_3, \ell_1^4, \ldots, \ell_{k_4}^4).
$$
By construction $a_1$ and $a_2$ are the minimum half-edges in $v_1$ and $v_2$ respectively.
However, $a_3$ is not necessarily minimal in $v_3$. If $a_3$ is the minimum, we have $a_1$,
$a_2$ and $a_3$ as the minimum half-edges in $v_1$, $v_2$ and $v_3$, respectively. Otherwise,
$\tau$ remains a trisection in $v_3$.

Consequently, we have two mappings, depending on whether or not $a_3$ is minimal:
$$
\rho_1 \colon (\overline{\mathfrak{m}}, \tau) \rightarrow (\mathfrak{m}, v_1, v_2, v_3),
\quad
\rho_2 \colon (\overline{\mathfrak{m}}, \tau) \rightarrow (\mathfrak{m}, v_1, v_2, \tau),
$$
where $\mathfrak{m}$, $\overline{\mathfrak{m}}$ are unicellular maps of genus $g$ and $g+1$,
respectively.

In the first case, $\tau$ is no longer a trisection after slicing and called a
{\it Type I}. In the second case, $\tau$ remains a trisection, a trisection of {\it Type II}.
\begin{proposition}\cite{Chapuy:11}
The mappings $\rho_1$ and $\rho_2$ are bijections.
\end{proposition}
Gluing can be described as follows:
given a unicellular map of $\mathfrak{m}_{g-k}$, together with a sequence of
vertices $V=\{v_1, \ldots v_{2k+1}\}$, where $v_i<_{\gamma} v_{i+1}$, $\forall 1\le i <2k+1$,
then: \\
{\bf I.} we glue the last three vertices $v_{2k-1}$, $v_{2k}$ and $v_{2k+1}$ via $\rho_1^{-1}$, thereby
obtaining the unicellular map $\mathfrak{m}_{g-k+1}$ together with a type I trisection $\tau^I$. \\
{\bf II.} we apply $\rho_2^{-1}(\mathfrak{m}_{g-k+i}, v_{2k-2i-1}, v_{2k-2i}, \tau^I)$ $k-1$ times for
$i=1$ to $i=k-1$. This produces the unicellular map $\mathfrak{m}_g(n)$, together with a
trisection $\tau^{II}$. The process defines a mapping
$$
\Lambda(\mathfrak{m}_{g-k}, v_1, \ldots, v_{2k+1})=(\mathfrak{m}_g, \tau).
$$
The order of the vertices in $V$ is induced by the boundary component, $\gamma$. Thus $V$ can be
considered as a set of vertices in $\mathfrak{m}_{g-k}$, ordered by $<{\gamma}$. $\Lambda$ merges
vertices from right to left by first applying $\Phi$ once then applying $\Psi$ until all vertices
are glued together.

$\Lambda$ is reversed as follows: given a unicellular map $\mathfrak{m}_g$ of genus $g$
and $i=0$: \\
{\bf 1.} if $\tau$ is type II trisection in $\mathfrak{m}_{g-i}$, then let $(\mathfrak{m}_{g-i-1},
v_{2i+1},v_{2i+2}, \tau)=\rho_2(\mathfrak{m}_{g-i},\tau)$. We increase $i$ to $i+1$ and repeat step
{\bf 1}. \\
{\bf 2.} if $\tau$ has type I, let $(\mathfrak{m}_{g-i}, v_{2i+1},v_{2i+2}, v_{2i+3})=\rho_1^{-1}(
         \mathfrak{m}_{g-i-1},\tau)$. \\
Then we return
$$
\Xi(\mathfrak{m}_g ,\tau)=(\mathfrak{m}_{g-i}, V_{\tau}).
$$
By construction, $\Lambda$ and $\Xi$ are inverse to each other.

The bijections $\Lambda$ and $\Xi$ immediately induce a connection between
unicellular maps having higher genus with those of lower genus.

\begin{theorem} \cite{Chapuy:11} \label{T:bij}
Let $U_g^t$ denote the set of tuples $(\mathfrak{m}_g, v_1, \ldots, v_{t})$,
where $v_1, \ldots, v_{t}$ is a sequence of vertices in $\mathfrak{m}_g$.
Furthermore, let $D_g$ denote the set of tuples $(\mathfrak{m}_g, \tau)$,
where $\tau$ is a trisection of $\mathfrak{m}_g$.
Then
$$
\Lambda\colon {\dot\bigcup}_{k=0}^{g-1} U_k^{2g-2k+1} \rightarrow D_g,
\quad
\Xi\colon D_g \rightarrow {\dot\bigcup}_{k=0}^{g-1} U_k^{2g-2k+1},
$$
are bijections and $\Lambda\circ \Xi=\text{\rm id}$ and $\Xi\circ \Lambda=\text{\rm id}$.
\end{theorem}

The theorem has the following enumerative corollary: let $\epsilon_{g}(n)$ denote the number of
unicellular map of genus $g$ having $n$ edges. Then
\begin{corollary}
\begin{equation} \label{E:induction1}
2g\cdot \epsilon_{g}(n) = {n+1-2(g-1) \choose 3} \epsilon_{g-1}(n)
+ \cdots + {n+1 \choose 2g+1} \epsilon_{0}(n).
\end{equation}
\end{corollary}
Here the $2g$-factor on left hand side counts the number of trisection in $\mathfrak{m}_g$ and
the binomial coefficients on the right hand side count the number of distinct
selections of subsets of $(2k+1)$ vertices from a unicellular map $\mathfrak{m}_{g-k}$.

Iterating $\Xi$, we obtain
\begin{equation} \label{E:induction2}
\epsilon_{g}(n)  = \sum_{0=g_0<g_1<\cdots<g_r=g} \prod_{i=1}^r
\frac{1}{2g_i}{n+1-2g_{i-1} \choose 2(g_i-g_{i-1})+1} \cdot \epsilon_{0}(n),
\end{equation}
where $\epsilon_0(n)$ is the number of planar trees having $n$ edges, i.e.~the Catalan
number $\frac{1}{n+1}{2n \choose n}$.


\section{Unicellular maps of genus $g$}\label{S:ucmap}

In Section~\ref{S:basic}, vertices are labeled with respect to only one iteration. After
applying $\Lambda$ there is a normalization via the factor $2g$ after which a new
labeling is being employed. In this Section we consider a pair consisting of a tree with
fixed labeled vertices and a unicellular map of genus $g$, also with a fixed set of
labelled vertices. We then study the set of glue paths from this tree recruiting
exclusively its labelled vertices, which produce the labelled unicellular map.

We begin by considering trees having $n$ edges and $k$ labeled vertices,
$\epsilon_{0}^{(k)}(n)$.
Clearly, the number of these trees is given by the Catalan number ${\rm Cat}(n)=
\frac{1}{n+1}{2n\choose n}$, i.e.~
$$
\epsilon_{0}^{(k)}(n) = {n+1 \choose k} \epsilon_{0}(n) = {n+1 \choose k}  {\rm Cat}(n).
$$

Next we study the case where $g>0$. Consider a unicellular map $\mathfrak{m}_{g,n}^{(k)}$
with $k$ labeled vertices. Applying the slicing bijection $\Xi$ once
we produce $2t+1$ labeled vertices and the genus decreases by $t$.
Therefore, we obtain a new unicellular map $\mathfrak{m}_{g-t,n}^{(k')}$ where $k'=k+2t+1$,
if in the former we slice an unlabeled vertex and $k'=k+2t$, if we slice a
labeled vertex. Then we have the following recursion
\begin{equation} \label{E:eg_rec}
2g\cdot \epsilon_{g}^{(k)}(n)= \sum_{t=1}^{g} {k+2t+1 \choose 2t+1} \epsilon_{g-t}^{(k+2t+1)}(n) +
\sum_{t=1}^{g} {k+2t \choose 2t+1} \epsilon_{g-t}^{(k+2t)}(n).
\end{equation}

Suppose we are given a tree $\mathfrak{m}_{0,n}^{(k_0)}$ having $k_0$ labeled vertices
and a unicellular map $\mathfrak{m}_{g,n}^{(k)}$ having $k$ labeled vertices. In order to
construct glue paths from $\mathfrak{m}_{0,n}^{(k_0)}$ to $\mathfrak{m}_{g,n}^{(k)}$ we proceed
by induction on $g$. The induction basis is clear and by induction hypothesis we have
obtained a labeled unicellular map $\mathfrak{m}_{g_1,n}^{(k_1)}$.
The map $\mathfrak{m}_{g_1, n}^{(k_1)}$ can produce two different labeled, unicellular maps,
namely $\mathfrak{m}_{g_2}^{(k_1-2(g_2-g_1)+1)}$ or $\mathfrak{m}_{g_2}^{(k_1-2(g_2-g_1))}$,
depending whether we label the new vertex or not.
By applying eq.~\ref{E:eg_rec}, we can compute the number of $\mathfrak{m}_{g,n}^{(k)}$
by $\mathfrak{m}_{0,n}^{(k_0)}$ inductively.

Let us first apply the new recursion in order to derive expressions for the generating
function of RNA structures of fixed topological genus,
$C_g^{(k)}(z)=\sum_{i=0}^{\infty} \epsilon_{g}^{(k)}(i) z^i$.

First we consider the case when $g=0$ and $k=0$, i.e., a tree without any labeled
vertex. Clearly, $C_0^{(0)}(z)$ satisfies
$$
C_0^{(0)}(z) = z(C_0^{(0)}(z))^2 + 1,
$$
whence $C_0^{(0)}(z) = \frac{1-\sqrt{1-4z}}{2z}$. For $k>0$, we accordingly have:

\begin{lemma} \label{L:label0}
We have
\begin{equation} \label{E:label0}
C_0^{(k)}(z) = {\rm Cat}(k-1)z^{k-1}(1-4z)^{-\frac{2k-1}{2}}, \forall k> 0,
\end{equation}
where ${\rm Cat}(n)$ denotes the Catalan number ${\rm Cat}(n)=\frac{1}{n}{2n \choose n}$.
\end{lemma}

\begin{proof}
A unicellular map of genus $0$, $\mathfrak{m}_{0,n}^{(k)}$, is a planar tree with $k$
labeled vertices. We decompose $\mathfrak{m}_{0,n}^{(k)}$ starting from its root.
Suppose $v$ is the first vertex we encounter and $e$ is the leftmost edge of $v$.
Removing $e$ we obtain two subtrees, containing $k_1$ and $k_2$ labeled vertices,
respectively, where $k_1+k_2=k$.
Therefore, the generating function $C_0^{(k)}(z)$ satisfies
\begin{equation} \label{E:label_rec}
C_0^{(k)}(z) = \sum_{i=0}^k z\cdot C_0^{(i)}(z)\cdot C_0^{(k-i)}(z), \quad k>1,
\end{equation}
and
$$
C_0^{(1)}(z) = 1+ z\cdot C_0^{(0)}(z)\cdot C_0^{(1)}(z) + z\cdot C_0^{(1)}(z)\cdot C_0^{(0)}(z).
$$
from which $C_0^{(1)}(z) = (1-4z)^{-1/2}$ follows. Furthermore, we observe that
$C_0^{(0)}(z)$ and $C_0^{(1)}(z)$ satisfy eq.~(\ref{E:label0}).

We continue by induction on $k$.
By induction hypothesis, for $1<t<k-1$, $C_0^{(t)}(z)$ satisfies eq.~(\ref{E:label0}).
Then solving eq.~(\ref{E:label_rec}) yields
$$
C_0^{(k)}(z) = \frac{1}{1-2zC_0^{(0)}(z)} \sum_{i=1}^{k-1} z\cdot C_0^{(i)}(z)\cdot C_0^{(k-i)}(z).
$$
By assumption, we have $C_0^{(t)}(z) = {\rm Cat}(t-1) z^t (1-4z)^{-(2t-1)/2}$, $\forall t<k$, whence
\begin{eqnarray*}
C_0^{(k)}(z) & = & \frac{1}{1-2zC_0^{(0)}(z)} z^{k-1}(1-4z)^{-(2k-2)/2} \sum_{i=1}^{k-1}
{\rm Cat}(i-1){\rm Cat}(k-i-1) \\
& = & z^{k-1}(1-4z)^{-(2k-1)/2} {\rm Cat}(k-1).
\end{eqnarray*}
Thus $C_0^{(k)}(z)$ also satisfies eq.~(\ref{E:label_rec}) and the lemma holds by induction.
\qed
\end{proof}

In view of eq.~(\ref{E:eg_rec}) and $C_g^{(k)}(z) = \sum_{i=0}^{\infty} \epsilon_g^{(k)}(i) z^i$
we derive
\begin{equation} \label{E:cg_rec}
2g\cdot C_g^{(k)}(z)= \sum_{t=1}^{g} {k+2t+1 \choose 2t+1} C_{g-t}^{(k+2t+1)}(z) +
\sum_{t=1}^{g} {k+2t \choose 2t+1} C_{g-t}^{(k+2t)}(z).
\end{equation}
Iterating the recursion of eq.~(\ref{E:cg_rec}) $r$ times we obtain a sequence of
tuples $(g_i, k_i)_{1\le i\le r}$, where $g_i$ is the genus of $\mathfrak{m}_{0,n}^{k_i}$ and
$k_i$ is the respective number of labeled vertices.
By construction, we have $k_i-k_{i-1}=2(g_i-g_{i-1})+1$, if the new vertex from gluing
is not labeled, and $k_i-k_{i-1}=2(g_i-g_{i-1})$, if the new vertex is labeled. Or put
differently, whether or not we sliced an unlabelled or a labeled vertex.
Let $t_i=k_i-k_{i-1}$, then the key information is expressed via the two sequences of
integers:
$$
g_0=g_0<g_1<\ldots<g_r=g, \quad \quad 0=t_0=t_1\le t_2\le \cdots \le t_r=r-t,
$$
where $r$ equals the number of applications of the mapping $\Xi$, $g_i$ is the genus
of $\mathfrak{m}_{g_i,n}$. The number $t_{i+1}-t_i$ is a signature indicating whether
we label new vertex or not. In case of $t_{i+1}-t_i=1$ we label the newly obtained
vertex in the $i$th step of gluing, and in case of $t_{i+1}-t_i=0$ we do not.
Accordingly a glue path between $\mathfrak{m}_{0,n}^{(k)}$ and $\mathfrak{m}_{g,n}^{(0)}$
can be reconstructed from the the sequence of pairs $(g_i, t_i)_{1\le i\le r}$.

We next employ this construction in order to express the generating function of
unicellular maps, $C_g(z)$ as follows:
\begin{theorem}\label{T:1.0}
The generating function of unicellular map of genus $g$ has the form
\begin{equation}
C_g(z) = \sum_{t=0}^{g-1}\kappa_t^{(g)}\cdot  \frac{z^{2g+t}}{(1-4z)^{2g+1+t-\frac{1}{2}}},
\end{equation}
where $\kappa_t^{(g)} = a_t^{(g)} {\rm Cat}(2g+t)$ and
\begin{equation}\label{E:at}
a_t^{(g)}= \sum_{0=g_0<g_1<\cdots < g_r=g \atop 0=t_0=t_1\le t_2\le \cdots \le t_r=r-t}
\prod_{i=1}^r \frac{1}{2g_i} {2g+t-(2g_{i-1}+(i-1))+t_i \choose 2(g_i-g_{i-1})+1}.
\end{equation}
\end{theorem}
We present the the coefficients $\kappa_t^{(g)}$ for genera $g\le 5$ in
Table~\ref{T:kappa}.
\begin{table}
\begin{center}
\begin{tabular}{c|ccccc}
 & g=1 & 2 & 3 & 4 & 5 \\
\hline
  t=0 & 1 & 21 & 1485 & 225225 & 59520825 \\
  1 &  & 105 & 18018 & 4660227 & 1804142340 \\
  2 &  &  & 50050 & 29099070 & 18472089636 \\
  3 &  &  &  & 56581525 & 78082504500 \\
  4 &  &  &  &  &  117123756750
\end{tabular}
\end{center}
\caption{\small Theorem~\ref{T:1.0}: the coefficients $\kappa_t^{(g)}$.
}\label{T:kappa}
\end{table}
\begin{proof}
Using eq.~(\ref{E:cg_rec}), we have
\begin{eqnarray*}
C_g(z) &=& \frac{1}{2g} \sum_{g_1=0}^{g-1} C_{g_1}^{2(g-g_1)+1}(z) \\
& = & \frac{1}{2g} \sum_{g_1=0}^{g-1} \frac{1}{2g_1}
\sum_{g_1>g_2} ( {2(g-g_2)+2 \choose 2(g_1-g_2)+1}
 C_{g_2}^{2(g-g_2)+2}(z) + {2(g-g_2)+1 \choose 2(g_1-g_2)+1}
 C_{g_2}^{2(g-g_2)+1}(z)  ) \\
&\vdots &\\
&=& \sum_{t=0}^{g-1} a_t^{(g)} C_{0}^{2g+t+1}(z)) \\
&=& \sum_{t=0}^{g-1} a_t^{(g)} {\rm Cat}(2g+t) z^{2g+t} (1-4z)^{-(2g+t+1-\frac{1}{2})} \\
&=& \sum_{t=0}^{g-1} \kappa_t^{(g)} z^{2g+t} (1-4z)^{-(2g+t+1-\frac{1}{2})},
\end{eqnarray*}
whence the theorem. \qed
\end{proof}

In \cite{Reidys:top1}, the generating function $C_g(z)$ has been shown to have the form
$$
C_g(z)=\frac{P_g(z)}{(1-4z)^{3g-\frac{1}{2}}},
$$
where $P_g(z)$ is a certain  polynomial.

In view of Theorem~\ref{T:1.0} we have the following expression for the $P_g(z)$ polynomials:
\begin{corollary}
The polynomial $P_g(z)$ is given by
\begin{equation}
P_g(z) = \sum_{t=0}^{g-1} \kappa_t^{(g)} (1-4z)^{g-1-t}.
\end{equation}
\end{corollary}

\section{Shapes of fixed genus}\label{S:shape}

In this section we study shapes of fixed topological genus. Since there are only finitely
many shapes for fixed genus $g$ \cite{Reidys:11a}, their generating function is a polynomial.
We give an explicit formula for the coefficients of the shape-polynomial, in which the same
$\kappa^{(g)}_i$ coefficients appear as in the generating function of unicellular maps of
genus $g$ in Theorem~\ref{T:1.0}.

We have shown in Section~\ref{S:basic} that a shape corresponds to a unicellular map in
which each vertex has degree greater than three, $\mathfrak{s}_{g,n}$. Applying $\Xi$ iteratively
to $\mathfrak{s}_{g,n}$ we derive a tree. By construction, any unlabeled vertex in this tree
originally comes from $\mathfrak{s}_{g,n}$ and thus retains its degree.

Let $\mathbb{S}^{(k)}_{g,n}$, denote the set of unicellular maps $\mathfrak{m}^{(k)}_g$ having $k$
labeled vertices, in which any unlabeled vertex has degree greater than or equal to three.
In particular, the set of unicellular maps corresponding to shapes of genus $g$ having $n$
edges is $\mathbb{S}^{(0)}_{g,n}$. In the following, let $\mathfrak{s}_{g,n}^{(k)}$ denote
an element in $\mathbb{S}^{(k)}_{g,n}$.

Since neither $\Lambda$ nor $\Xi$ alter unlabelled vertices we have induced bijections
$$
\Lambda\colon (\mathfrak{s}_{g-t,n}^{(k+2t+1)}, V_{2t+1}) \rightarrow (\mathfrak{s}_{g,n}^{(k)},
\tau)
\quad
\Xi\colon (\mathfrak{s}_{g,n}^{(k)}, \tau) \rightarrow  (\mathfrak{s}_{g-t,n}^{(k+2t+1)}, V_{2t+1}).
$$
Indeed, $\Xi$, slices a vertex together with a trisection into a sequence of labeled vertices.
and thus does not change the degree of unlabeled vertices in the map. Furthermore, $\Lambda$
glues three or more labeled vertices into one vertex, which has accordingly minimum degree
$3$.

Let $\eta_g(n,k)$ denote the cardinality of $\mathbb{S}^{(k)}_{g,n}$, $k\ge 0$. In view of the
above and eq.~(\ref{E:cg_rec}) we have
\begin{equation}\label{E:eta}
2g\cdot \eta_g(n,k) = \sum_{i=1}^g {k+2i+1 \choose 2i+1} \eta_{g-i}(n,k+2i+1) +
\sum_{i=1}^{g} {k+2i \choose 2i+1}\eta_{g-i}(n,k+2i),
\end{equation}
where ${n\choose m}=0$, for $n<m$.

Due to the compatibility of slicing and gluing with the vertex degree of unlabelled
vertices we can conclude

\begin{proposition}\label{P:1.1}
The number of shapes of genus $g$ is given by
\begin{equation}\label{E:kk}
\eta_g(n,0) = \sum_{t=0}^{g-1} a_t^{(g)} \eta_0(n,2g+t+1),
\end{equation}
where the $a_t^{(g)}$ are given by eq.~(\ref{E:at}).
\end{proposition}
\begin{proof}
Iterating the recursion in eq.~\ref{E:eta} we reduce the genus. In view of
$$
\eta_g(n,0) = \frac{1}{2g} \sum_{g_1=0}^{g-1} \eta_{g_1}(n,2(g-g_1)+1),
$$
we then substitute $\eta_{g_1}(n,2(g-g_1)+1)$ using eq.~(\ref{E:cg_rec}) and obtain
\begin{eqnarray*}
\eta_g(n,0) & = & \frac{1}{2g} \sum_{g_1=0}^{g-1} \frac{1}{2g_1}
\sum_{g_1>g_2}( {2(g-g_2)+2 \choose 2(g_1-g_2)+1}
\eta_{g_2}(n,2(g-g_2)+2)  \\
& + &
{2(g-g_2)+1 \choose 2(g_1-g_2)+1}
\eta_{g_2}(n,2(g-g_2)+1) ).
\end{eqnarray*}
Continuing this substitution we arrive at $\eta_0(n,w)$ for some integer $w$.
Since in each substitution, the number of labeled vertices increases by
either $2i$ or $2i+1$, we derive
$$
\eta_g(n,0) = \sum_{t=0}^{g-1} a_t^{(g)} \eta_0(n,2g+t+1),
$$
where the coefficients, $a_t^{(g)}$, are given by eq.~\ref{E:at}.
\end{proof}

At this point we observe that it possible to analyze the terms $\eta_0(n,2g+t+1)$
further. The idea is to ``remove'' all unlabeled vertices from any partially labelled
tree, thereby reducing the recursion to fully labelled trees. The latter are then
enumerated by Catalan numbers.

This removal is facilitated by observing that we can restrict the bijection of
R\'{e}my to $\mathbb{S}^{(k)}_{0,n}$-trees.

To this end, let us first recall R\'{e}my's bijection for planar trees \cite{Remy:85}.
\begin{theorem}\label{T:remy}
Let $\epsilon_0(n)$ denote the number of planar trees with $n$ edges.
Then we have the recursion
$$
(n+1) \epsilon_0(n) = 2(2n-1) \epsilon_0(n-1).
$$
\end{theorem}
The bijection of Theorem~\ref{T:remy} associates a planar tree having $n$ edges and
a labeled vertex to a planar tree with $(n-1)$ edges with a labeled sector.
It is constructed as follows: observing that in a planar tree having $n$ edges,
there are $n+1$ vertices and $2n-1$ sectors, R\'{e}my's bijection, illustrated in
Figure~\ref{F:remy}, entails two ways of inserting a vertex into a labeled sector.
This vertex-insertion generates from a planar tree with
$n-1$ edges and a labelled sector a planar tree with $n$ edges and a labelled vertex.
The process can be reversed, i.e., a planar tree with $n$ edges and a labeled vertex
can be re-tracked to a planar tree with $n-1$ edges and a labeled sector.
Depending on the labeled vertex being a leave or not, one derives a planar tree having
two types of labeled sectors.

\begin{figure}[ht]
\begin{center}
\includegraphics[width=0.8\columnwidth]{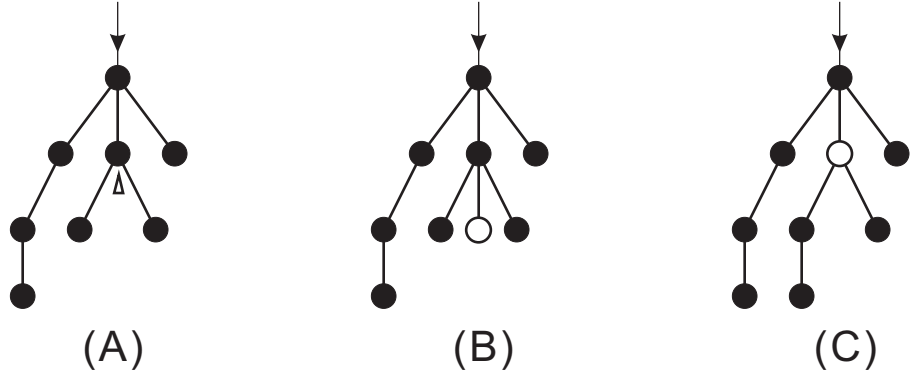}
\end{center}
\caption{\small  R\'{e}my's bijection: two ways of obtaining a planar tree with $n$ edges
and a labeled vertex from a planar tree with $n-1$ edges with a labeled sector. We pass from
(A) to (B) by inserting a labeled vertex as a leaf to the labeled sector and from (A) to (C)
by replacing the vertex containing the sector by the labeled vertex, and carrying the
subtree on the left of the sector as its leftmost subtree.
This case applies, if the labeled vertex is not a leave.
}
\label{F:remy}
\end{figure}

We shall prove that R\'{e}my's procedure contracts unlabeled vertices of a
$\mathbb{S}^{(k)}_{0,n+1}$-tree into a particular type of sector in the resulting
$\mathbb{S}^{(k)}_{0,n}$-tree. These sectors are referred to as {\it shape-sector}
and are defined as follows: suppose we are given a $\mathbb{S}^{(k)}_{0,n}$-tree.
A shape-sector is a sector for which R\'{e}my's procedure, inserting a non-leaf
unlabeled vertex, generates a $\mathbb{S}^{(k)}_{0,n+1}$-tree.

\begin{lemma}
A $\mathbb{S}^{(k)}_{0,n}$-tree contains exactly $(2k-n-2)$ shape-sectors.
\end{lemma}
\begin{proof}
In order to construct a $\mathbb{S}^{(k)}_{0,n+1}$ from a $\mathbb{S}^{(k)}_{0,n}$-tree
by R\'{e}my's procedure, we need to ensure that the newly inserted, unlabeled vertex
has at least degree $3$ and that it does not reduce the degree of the other unlabeled
vertex. We shall consider only the insertion not producing a leaf, since it is a
vertex of degree $1$.
Given a sector $\tau$ in a vertex $v$, assume the children
of $v$ are indexed counterclockwise $v_1\ldots v_m$, where $m\ge 2$.
The sector $\tau$ partitions the $v$-children
into two blocks:
$$
\{v_1 \ldots, v_\ell\}\quad  \text{\rm and}\quad
\{v_{\ell+1} \ldots, v_m\}.
$$
We make two observations: (a) the newly inserted vertex has at least degree three
if $\{v_{\ell+1} \ldots, v_m\}\neq \varnothing$, see Figure~\ref{F:shape preserved}.
(b) the vertex, that is pushed ``down'' by the newly inserted vertex retains degree
$\ge 3$, if and only if $\ell\ge 2$, i.e.~if $\tau$ is to the right of the second $v$-child
in counterclockwise order.\\
The latter applies by construction only to the case where the pushed-down vertex is
unlabeled.

\begin{figure}[ht]
\begin{center}
\includegraphics[width=0.9\columnwidth]{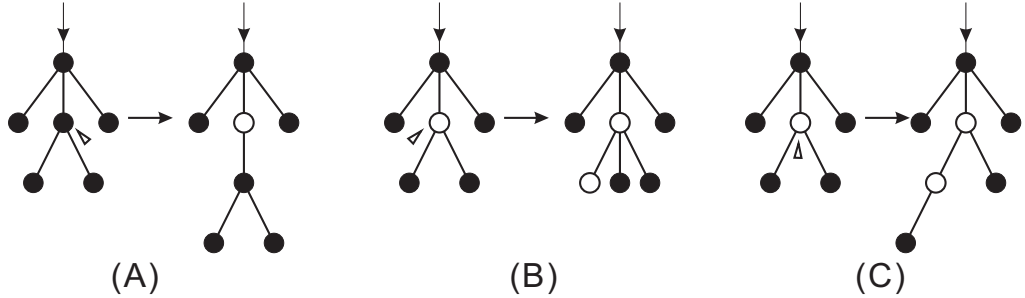}
\end{center}
\caption{\small How to construct non shape-sectors via R\'{e}my's-procedure.
The following insertions produce an unlabeled vertex of degree $2$:
(A) the sector is in a labeled vertex and $\{v_{\ell+1} \ldots, v_m\}= \varnothing$, and
(B) ($\ell=0$) and (C) ($\ell=1$) if the sector is in an unlabeled vertex.
}
\label{F:shape preserved}
\end{figure}

Therefore, a sector in a vertex $v$ is a shape-sector if and only if
\begin{itemize}
\item $\tau$ has the property $\{v_{\ell+1} \ldots, v_m\}\neq \varnothing$,
\item if $v$ is unlabeled, $\tau$ satisfies furthermore
      $\vert\{v_{1}, \ldots, v_\ell\}\vert\ge 2$.
\end{itemize}
There are in total $2n+1$ sectors in a tree. The first criterion rules
out $n+1$ sectors, since each vertex has one such sector. The second criterion
rules out $2$ sectors from an unlabeled vertices and there are $n+1-k$ of them.
Accordingly, the number of shape-sectors is given by
$$
2n+1 - (n+1) - 2(n+1-k) = 2k-n-2.
$$
Any of the $2k-n-2$ shape preserved sectors produces a $\mathbb{S}^{(k)}_{0,n+1}$-tree by
R\'{e}my's procedure, when inserting a non-leaf vertex and the lemma follows.
\end{proof}

\begin{corollary}\label{C:1.1}
Let $\tau$ denote a shape preserved sector and $\mathfrak{m} \in \mathbb{S}^{(k)}_{0,n}$.
Then
$$
\rho\colon (\mathfrak{m}, \tau) \rightarrow (\mathfrak{m}', v)
$$
is a bijection, where $v$ is an unlabeled vertex in $\mathfrak{m}'$ and $\mathfrak{m}' \in
\mathbb{S}^{(k)}_{0,n+1}$. In particular we have
\begin{eqnarray*}
(2k-n-2) \eta_0(n,k) & = &  (n+1-k) \eta_0(n+1,k) \\
\eta_0(n,k)= {2k-(k-1)-2 \choose n+1-k} \eta_0(k-1,k) & = & {k-1 \choose n+1-k} {\rm Cat}(k-1).
\end{eqnarray*}
\end{corollary}
\begin{proof}
The corollary follows by restriction of R\'{e}my's bijection. This bijection implies
that the removal of an unlabeled vertex of a $\mathbb{S}^{(k)}_{0,n}$-tree produces
a shape-sector. Furthermore, the order of such removals is irrelevant. Therefore, a
$\mathbb{S}^{(k)}_{0,n}$-tree can be constructed from a $\mathbb{S}^{(k)}_{0,m-1}$-tree together
with $(2k-n+1)$ shape preserved sectors. Clearly, the number of $\mathbb{S}^{(k)}_{0,k-1}$-trees
equals ${\rm Cat}(k-1)$, where ${\rm Cat}(n)$ is the $n$-th Catalan number given by
$\frac{1}{n+1}{2n \choose n}$. To obtain a $\mathbb{S}^{(k)}_{0,n}$-tree, we need to
insert $n+1-k$ unlabeled vertices. Choosing $n+1-k$ out of $2k-(k-1)-2$ shape preserved
sectors from $\mathbb{S}^{(k)}_{0,k-1}$, we derive
$$
\eta_0(n,k)= {2k-(k-1)-2 \choose n+1-k} \eta_0(k-1,k)
=  {k-1 \choose n+1-k} {\rm Cat}(k-1),
$$
whence the corollary.
\end{proof}

In particular, the number of shape-sectors decreases by $1$, upon insertion of one,
unlabeled vertex and there are at most $2k-(k-1)-2=k-1$ insertions into a fully
labeled tree having $n$ edges. This provides another proof that for fixed topological
genus there are only finitely many shapes.

We next compute the shape polynomial
$S_g(z)=\sum_n s_g(n)z^g$ where $s_g$ is the number of shapes having $n$ arcs.
Note that $s_g(n)=\eta_g(n,0)$.

\begin{theorem}\label{T:1.1}
The shape generating function is given by
\begin{equation}
S_g(z)= \sum_{t=0}^{g-1} \kappa_t^{(g)} z^{2g+t}(1+z)^{2g+t},
\end{equation}
where $\kappa_t^{(g)}=a_t^{(g)} {\rm Cat}(2g+t)$.
\end{theorem}
\begin{proof}
Since $s_g(n)=\eta_g(n,0)$ we have
$$
S_g(z)= \sum_n \eta_g(n,0) z^n = \sum_n (\sum_{t=0}^{g-1} a_t^{(g)} \eta_0(n,2g+t+1) ) z^n.
$$
By Corollary~\ref{C:1.1} we can express the terms $\eta_0(n,2g+t+1)$,
\begin{eqnarray*}
S_g(z) & = & \sum_n (\sum_{t=0}^{g-1} a_t^{(g)} \eta_0(n,2g+t+1) ) z^n \\
& = & \sum_n \sum_{t=0}^{g-1} a_t^{(g)} {2g+t \choose n-(2g+t)} {\rm Cat}(2g+t) z^n \\
&=& \sum_{t=0}^{g-1} \kappa_t^{(g)} \left(z^{2g+t}
\sum_{n=2g+t}^{2(2g+t)} {2g+t \choose n-(2g+t)} z^{n-(2g+t)} \right) \\
&=& \sum_{t=0}^{g-1} \kappa_t^{(g)} z^{2g+t}(1+z)^{2g+t},
\end{eqnarray*}
where $\kappa_t^{(g)}=a_t^{(g)} \cdot {\rm Cat}(2g+t)$.
\end{proof}

\section{Uniform generation}\label{S:Gen}

In this section we present an algorithm that generates shapes of fixed
genus $g$. Since the generating function of shapes is a polynomial,
a shape of fixed topological genus has only finitely many arcs.
In fact we have $2g\le n \le 3g-1$
The probability of having exactly $n$ arcs in a
shape of fixed genus $g$ is given by
\begin{equation}
\mathbb{P}_g(n) = \frac{s_g(n)}{\sum_{t=0}^{g-1}s_g(2g+t)},
\end{equation}
where $s_g(n)$ is the coefficients in $S_g(z)$.
In the following we generate shapes of fixed topological genus $g$ and fixed
number of arcs $n$, where $2g\le n \le 3g-1$.

In Section~\ref{S:shape} we have shown that a shape of genus $g$ is obtained
by gluing $k$ labeled vertices contained in a partially labelled tree having
$n$ edges, where $2g+1\le k\le 3g$. Furthermore, this partially labelled tree
is constructed by R\'{e}my's bijection restricted to shape-sectors.
The corresponding probability of this event reads
$$
\mathbb{P}_{g}(n,k) = \frac{a_{k-2g-1}^{(g)} \eta_0(n, k)}
                                 {\sum_{t=0}^{g-1}{a_t^{(g)} \eta_0(n, 2g+t+1)}}.
$$

The algorithm generates a shape of genus $g$ in three steps:
\begin{itemize}
\item first we uniformly generate a tree with $k$ vertices \cite{SecondaryStructureSampling},
\item secondly we insert $n+1-k$ unlabeled vertices by R\'{e}my's procedure in shape-sectors.
      Since all vertices in the tree are labeled, there are $k-1$ shape-sectors.
      We then uniformly select $n+1-k$ from the $k-1$ shape-sectors and insert
      unlabeled vertices by R\'{e}my's procedure. The probability of a particular
      selection is given by $1/{k-1 \choose n+1-k}$,
\item thirdly we select uniformly a particular glue path.
      In Section~\ref{S:shape} we established the trace, i.e.~a
      sequence of pairs $(g_i,t_i)$, such that $g_i<g_{i+1}$ and $t_{i+1}-t_i$
      being equal to either $0$ or $1$.
      Suppose the number of labeled vertices is $k$ and $t=k-2g-1$.
      We construct the sequence inductively, with initial status $(g_0=0, t_0=0)$
      and terminate in case of $(g_r=g, t_r=r-t)$ for some integer $r$.
      Let
      $$
      \mathbb{P}_{\text{step}}(i|i-1)= \frac{1}{2g_i} \cdot {2g+t-(2g_{i-1}+(i-1))+t_i \choose 2(g_i-g_{i-1})+1}.
      $$
      Then we have
\begin{eqnarray*}
      & & \mathbb{P}((g_i, t_i)|(g_{i-1},t_{i-1}))=  \\
      && \frac
      {\mathbb{P}_{\text{step}}(i|i-1) \cdot
      \sum_{g_{i+1}<\cdots < g_r=g \atop t_{i+1}\le \cdots \le t_r=r-t}
      \prod_{\ell=i+1}^r \frac{1}{2g_{\ell}}{2g+t-(2g_{\ell-1}+(\ell-1))+t_\ell \choose
      2(g_\ell-g_{\ell-1})+1}}
      {\sum_{g_i<\cdots < g_r=g \atop t_i\le \cdots \le t_r=r-t}
      \prod_{\ell=i}^r \frac{1}{2g_{\ell}}{2g+t-(2g_{\ell-1}+(\ell-1))+t_\ell \choose 2(g_\ell-g_{\ell-1})+1}}.
      \end{eqnarray*}
      We accordingly derive
      \begin{eqnarray*}
      & & \prod_{i=1}^r \mathbb{P}((g_i, t_i)| (g_{i-1},t_{i-1}))= \\
      & &\frac{\prod_{i=1}^r \mathbb{P}_{\text{step}}(i|i-1)}
      {\sum_{0=g_0<g_1<\cdots < g_r=g \atop 0=t_0=t_1\le t_2\le \cdots \le t_r=r-t}
      \prod_{i=1}^r \frac{1}{2g_{i}}{2g+t-(2g_{i-1}+(i-1))+t_i \choose 2(g_i-g_{i-1})+1}}.
      \end{eqnarray*}
\item finally we realize the $(g_i, t_i)_i$-glue path by selecting vertices.
      Suppose we have at step $i-1$ the labeled shape
      $\mathfrak{s}^{(k_{i-1})}_{g_{i-1},n}$.
      Then we select $2(g_{i}-g_{i-1})+1$ from the $k_{i-1}$ labeled vertices uniformly
      and glue them via $\Lambda$. There are ${2g+t-2(g_{i-1}+(i-1))+t_i \choose
      2(g_i-g_{i-1})+1}$ ways to choose these vertices uniformly, which generates the
      $\mathfrak{s}^{(k_{i})}_{g_{i},n}$ together with a labeled trisection.
      Since there are exactly $2g_{i}$ trisections in $\mathfrak{s}^{(k_{i})}_{g_{i},n}$,
      there are $2g_{i}$ glue paths generating the same configuration.
      Therefore, erasing the label of the trisections induces each glue path with
      a $1/2g_{i}$ factor.  Accordingly there are
$$
\prod_{i=1}^r \mathbb{P}_{\text{step}}(i|i-1)
$$
      paths with trace $((g_i,t_i))_{i=1}^r$ from a fixed $\mathfrak{s}^{(2g+t)}_{0,n}$ to
      $\mathfrak{s}^{(0)}_{g,n}$.
     \begin{eqnarray*}
      \frac{\prod_{i=1}^r \mathbb{P}((g_i, t_i)| (g_{i-1},t_{i-1}))}{
     \prod_{i=1}^r \mathbb{P}_{\text{step}}(i|i-1)}
      = \frac{1}{a_t^{(g)}}.
      \end{eqnarray*}
\end{itemize}

The above process can be formally expressed as follows:
\begin{algorithm}
\begin{algorithmic}[1]
\STATE {\tt UniformShape}~($TargetGenus$)
\STATE {$n \leftarrow {\tt NumberOfArcs}(g)$}
\STATE {$k \leftarrow {\tt NumberOfLabel}(n, g)$}
\STATE {$t \leftarrow k-2g$}
\STATE {$\mathfrak{s}_{0,k-1} \leftarrow {\tt UnifomTree}(k-1)$}
\STATE {$\mathfrak{s}_{0,n}^{(2g+t)} \leftarrow {\tt InsertUnlabeled}(\mathfrak{s}_{0,k-1}, n-k+1)$}
\STATE {$i\leftarrow 1$}
\STATE {$(g_0, t_0) \leftarrow (0,0)$}
\WHILE {$g_i \le TargetGenus$}
\STATE {$(g_{i}, t_{i}) \leftarrow {\tt NextGenus}~((g_j, t_j)_{0\le
j < i}, TargetGenus)$}
\STATE {$i\leftarrow i+1$}
\ENDWHILE
\STATE {$i\leftarrow 1$}
\WHILE {$g_i \le TargetGenus$}
\STATE {$V_{i-1} \leftarrow {\tt SelectVertex}~(\mathfrak{s}^{(2g+t-(2g_{i-1}+(i-1))+t_{i})}_{g_{i-1} ,n},
2(g_{i}-g_{i-1})+1)$}
\STATE {$\mathfrak{s}^{(2g+t-(2g_{i}+(i))+t_{i+1})}_{g_{i} ,n}
\leftarrow
{\tt Glue}~(\mathfrak{s}^{(2g+t-(2g_{i-1}+(i-1))+t_i)}_{g_{i-1} ,n}, V_i,
t_{i+1})$}
\STATE {$i\leftarrow i+1$}
\ENDWHILE
\STATE \textbf{return} $\mathfrak{s}^{(0)}_{g,n}$
\caption {\small }
\label{A:shape}
\end{algorithmic}
\end{algorithm}

The next proposition is implied by Proposition~\ref{P:1.1} and Corollary~\ref{C:1.1}:

\begin{proposition}
The probability of a shape generated by Algorithm~\ref{A:shape} is $1/\eta_g(n,0)$, i.e.~the
algorithm generates shapes of genus $g$ uniformly.
\end{proposition}
\begin{proof}
By construction, the probability of arriving at $\mathfrak{s}^{(0)}_{g,n}$ from
$\mathfrak{s}^{(2g+t)}_{0,n}$ is given by $1/a_t^{(g)}$.
The probability to glue from a $\mathfrak{s}^{(2g+t)}_{0,n}$-shape equals
$\mathbb{P}_{g, n}(k)$ and is a result of eq.~(\ref{E:kk}).
In view of Corollary~\ref{C:1.1}, which expresses $\eta_0(n,k)$
as ${k-1 \choose n+1-k} {\rm Cat}(k-1)$, the probability of a particular shape
$\mathfrak{s}^{(0)}_{g,n}$, generated by Algorithm~\ref{A:shape} is given by
\begin{eqnarray*}
\mathbb{P}(\mathfrak{s}^{(0)}_{g,n}) & = &\mathbb{P}_{g, n}(k)\cdot
\frac{1}{{\rm Cat}(k-1)} \cdot  \frac{1}{{k-1 \choose n+1-k}} \cdot
\frac{1}{a_t^{(g)}} \\
& = & \frac{a_t^{(g)} \eta_0(n, k)} {\sum_{t=0}^{g-1}{a_t^{(g)} \eta_0(n, 2g+t+1)}} \cdot
\frac{1}{{\rm Cat}(k-1){k-1 \choose n+1-k}} \cdot \frac{1}{a_t^{(g)}} \\
& = & \frac{1}{\sum_{t=0}^{g-1}{a_t^{(g)} \eta_0(n, 2g+t+1)}} \\
& = & \frac{1}{\eta_g(n,0)}.
\end{eqnarray*}
\end{proof}


\section{Discussion}


In this paper, we studied shapes of RNA structures. While topologically motivated by taking
homotopy classes of arcs, shapes have a simple combinatorial interpretation and there are,
for fixed genus, only finitely many of them.

Topological folding algorithms like \cite{Reidys:11a} show that shapes determine the
grammar of the multiple-context free language of topological RNA structures of
fixed topological genus. It is therefore of interest to compute shapes of genus $g$,
effectively. This connection also makes concrete how the topology of RNA structures
characterizes their language and reiterates the fact that RNA shapes carry genuinely
key information of RNA structures.

Two questions immediately arise. First, can we compute the generating polynomial for arbitrary
topological genus with preferably explicit expressions for the coefficients. The latter are
of significance as they represent the number of shapes of genus $g$ with $n$ arcs. Secondly,
can we actually generate these shapes--i.e.~how do they look like as diagrams? This allows
us to derive a plethora of nontrivial statistics of shape. Both questions are answered
affirmatively in this paper.

As for the first, in Section~\ref{S:shape}, Theorem~\ref{T:1.1} we compute the shape polynomial for
fixed topological genus. As for the second, we specify in Section~\ref{S:Gen} an algorithm that
uniformly generates shape of fixed genus $g$ in linear time. We also implement
Algorithm~\ref{A:shape} and its source code is available at
\begin{center}
\url{http://imada.sdu.dk/~duck/shape.c}
\end{center}

To illustrate uniformity, we display in Figure~\ref{F:uniformshape} the multiplicities of shapes
of genus $2$ obtained by Algorithm~\ref{A:shape} and the Binomial coefficients
${N \choose \ell}(1/\sigma_2)^\ell(1-1/\sigma_2)^{N-\ell}$.

\begin{figure}[ht]
\begin{center}
\includegraphics[width=0.55\columnwidth]{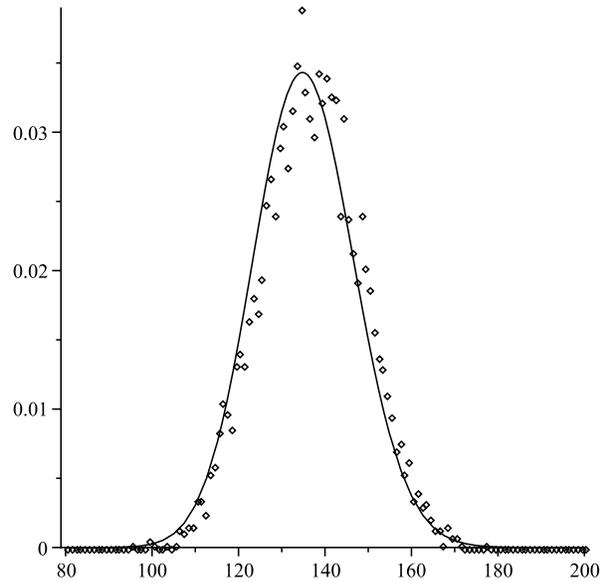}
\end{center}
\caption{\small Uniform generation of shapes:
We generate $N=5\times 10^5$ shapes and display the frequencies of their multiplicities
(dots) together with the Binomial coefficients of the uniform sampling
${N \choose \ell}(1/\sigma_2)^\ell(1-1/\sigma_2)^{N-\ell}$ (solid curve),
where $\sigma_2=3696$ is the number of shapes of genus $2$.
}
\label{F:uniformshape}
\end{figure}

We next discuss how to use shapes in order to extract key information from databases.
Let us begin with an experiment: we uniformly sample RNA structures of length $200$
having genus $2$ and study the frequencies of their associated shapes.
We observe that first shapes of the same length remains uniformly distributed, see
Figure\ref{F:structure-shape} (A).
Secondly, the distribution of shapes of different length follows the distribution
of the coefficients in the shape polynomial, see Figure~\ref{F:structure-shape} (B).

\begin{figure}[ht]
\begin{center}
\includegraphics[width=0.85\columnwidth]{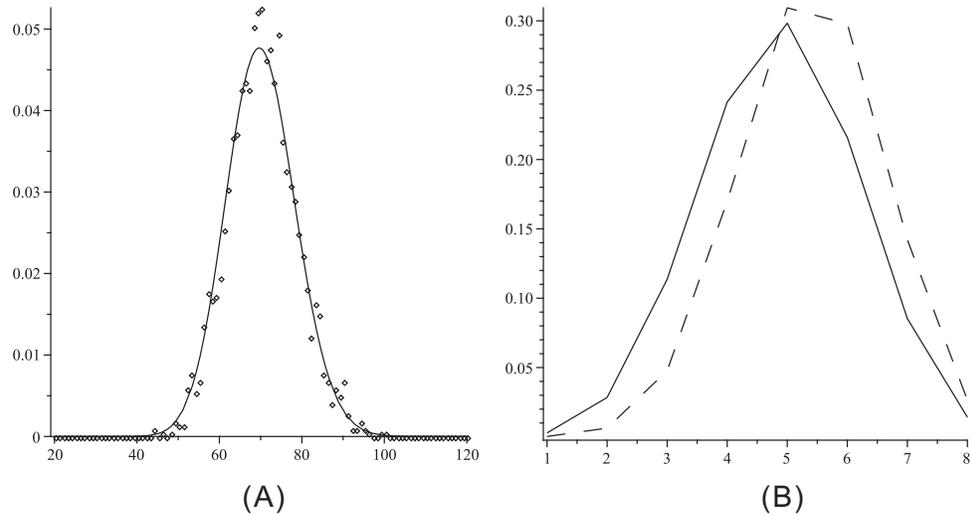}
\end{center}
\caption{\small Uniform sampling of structures of length $200$ of genus $2$,
(over $5\times 10^5$ in total) and display: the multiplicities of shapes with length
$16$ (A) and the multiplicities of shapes of different lengths (B).
The solid curve displays the distribution induced by the coefficients of the shape
polynomial, while the dash curve displays the distribution of the sampling process.
}
\label{F:structure-shape}
\end{figure}

This observation motivates to extract such shape multiplicities from a database of
RNA-structures and to use this in order to generate RNA structures from shapes using
the latter, see Figure~\ref{F:database}, where we display the number of RNA pseudoknot
structures (PKB1--PKB304) found in Pseudobase \cite{Taufer:09} having a particular shape.
This idea allows to reduce database information succinctly in form of a novel
polynomial whose coefficients are the multiplicities of shapes of fixed length.
In particular, in case of uniform RNA structures this method would recover the shape
polynomial itself.

\begin{figure}[ht]
\begin{center}
\includegraphics[width=0.9\columnwidth]{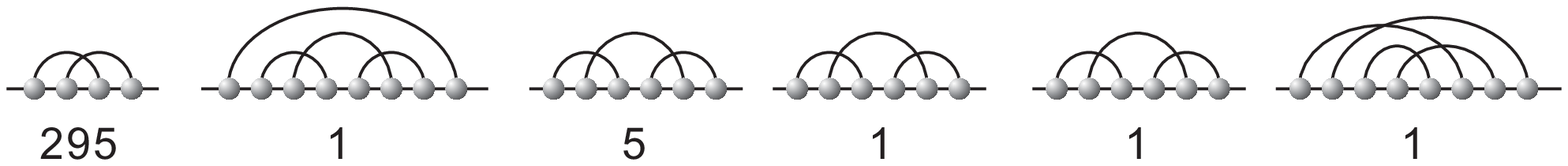}
\end{center}
\caption{\small Shape-multiplicities of Pseudobase PKB1--PKB304 \cite{Taufer:09}.
}
\label{F:database}
\end{figure}

While the uniform generation algorithm of shapes is of best possible time complexity, it
does generate, strictly speaking, {\it labeled} shapes. That is, shapes with a
distinct trisection. Our next objective is to work on obtaining fully bijective construction
methods for unlabeled shapes, i.e.~explicit algorithms that derive inductively shapes without
encountering the multiplicity $2g$. Ultimately this does not matter for applications
(nor the time complexity) of the uniform generation {\it per se} since each labeled shape is
generated with the same, finite, multiplicity ($2g$). However, it would be interesting to identify
a construction method for unlabeled shapes.

\section{ Acknowledgments.}
We acknowledge the financial support of the Future and Emerging
Technologies (FET) programme within the Seventh Framework Programme (FP7) for
Research of the European Commission, under the FET-Proactive grant agreement
TOPDRIM, number FP7-ICT-318121.

\bibliographystyle{plain}
\bibliography{shape}

\end{document}